\documentclass[a4paper,12pt]{article}
\usepackage{amsmath}
\usepackage{amsthm}
\usepackage{amssymb}
\usepackage{enumerate}
\usepackage{url}

\title{A variational representation and Pr\'ekopa's theorem for 
Wiener functionals}

\author{Yuu Hariya\thanks{Mathematical Institute, 
Tohoku University, Aoba-ku, Sendai 980-8578, Japan. }}

\date{\empty}
\numberwithin{equation}{section}

\theoremstyle{plain}

\newtheorem{thm}{Theorem}[section]

\newtheorem{prop}{Proposition}[section]

\newtheorem{lem}{Lemma}[section]

\theoremstyle{definition}

\theoremstyle{remark}
\newtheorem{rem}{Remark}[section]
\newtheorem{exm}{Example}[section]

\DeclareMathOperator*{\esssup}{ess\,sup}

\begin{document}

\def\N {\mathbb{N}}
\def\R {\mathbb{R}}
\def\Q {\mathbb{Q}}

\def\calF {\mathcal{F}}

\def\kp {\kappa}

\def\ind {\boldsymbol{1}}

\def\al {\alpha }
\def\la {\lambda }
\def\ve {\varepsilon}
\def\Om {\Omega}

\def\v {v}

\def\ga {\gamma }

\def\W {\mathbb{W}}
\def\H {\mathbb{H}}
\def\A {\mathcal{A}}

\newcommand\ND{\newcommand}
\newcommand\RD{\renewcommand}

\ND\lref[1]{Lemma~\ref{#1}}
\ND\tref[1]{Theorem~\ref{#1}}
\ND\pref[1]{Proposition~\ref{#1}}
\ND\sref[1]{Section~\ref{#1}}
\ND\ssref[1]{Subsection~\ref{#1}}
\ND\aref[1]{Appendix~\ref{#1}}
\ND\rref[1]{Remark~\ref{#1}} 
\ND\cref[1]{Corollary~\ref{#1}}
\ND\eref[1]{Example~\ref{#1}}
\ND\fref[1]{Fig.\ {#1} }
\ND\lsref[1]{Lemmas~\ref{#1}}
\ND\tsref[1]{Theorems~\ref{#1}}
\ND\dref[1]{Definition~\ref{#1}}
\ND\psref[1]{Propositions~\ref{#1}}
\ND\rsref[1]{Remarks~\ref{#1}}
\ND\sssref[1]{Subsections~\ref{#1}}

\ND\pr{\mathbb{P}}
\ND\ex{\mathbb{E}}
\ND\br{W}

\ND\be[1]{R^{(#1)}}

\ND\E[1]{\mathcal{E}^{#1}}
\ND\no[2]{\|{#1}\|_{#2}}
\ND\nt[2]{\|{#1}\|_{#2}}
\ND\tr[2]{T^{#1}_{#2}}
\ND\si{\mathcal{S}}
\ND\lhs[1]{\log\ex\left[e^{{#1}(\br)}\right]}
\ND\inner[2]{({#1},{#2})_{\H }}
\ND\cpl[2]{\langle {#1},{#2}\rangle}

\def\thefootnote{{}}

\maketitle 

\begin{abstract}
In 1998, Bou\'e and Dupuis proved a variational representation for 
exponentials of bounded Wiener functionals. Since their proof 
involves arguments related to the weak convergence 
of probability measures, the boundedness of functionals seems inevitable. 
In this paper, we extend the representation to unbounded functionals 
under a mild assumption on their integrability. 
As an immediate application of the extension, we prove an analogue of 
Pr\'ekopa's 
theorem for Wiener functionals, which is then applied to formulate 
the Brascamp-Lieb inequality in the framework of Wiener spaces. 
\footnote{E-mail: hariya@math.tohoku.ac.jp}
\end{abstract}

\section{Introduction and main results}\label{;intro}
Let $\br $ be a standard $d$-dimensional Brownian motion. In \cite{bd} 
Bou\'e and Dupuis showed the following representation for 
any bounded and measurable functional $F$ that maps $C([0,1];\R ^{d})$ 
into $\R $: 
\begin{align}\label{;vr0}
 \log \ex \left[ 
 e^{F(W)}
 \right] 
 =\sup _{\v }\ex \left[ 
 F\left( \br +\int _{0}^{\cdot }\v _{s}\,ds\right) -\frac{1}{2}\int _{0}^{1}|\v _{s}|^{2}\,ds
 \right] , 
\end{align}
where the expectation $\ex $ is relative to $\br $ and the supremum is over 
all processes that are progressively measurable with respect 
to the augmentation of the natural filtration of $W$. In \cite{bd} 
the variational representation \eqref{;vr0} was proven to be useful in deriving 
various large deviation asymptotics such as Laplace principles for small noise 
diffusions described by stochastic differential equations. 
These results have been extended by Budhiraja and Dupuis 
\cite{bud} to Hilbert space-valued Brownian motion, and later generalized 
by Zhang \cite{zha} to the framework of abstract Wiener spaces. 
In Bou\'e-Dupuis \cite{bd2}, the representation \eqref{;vr0} is 
also applied to risk-sensitive stochastic control problems. 

One of the purposes of this paper is to extend the 
representation \eqref{;vr0} to any unbounded functional 
$F$ that satisfies a certain integrability condition; 
the condition we impose is reasonably weak so that it allows 
$F$ to diverge to $-\infty $ exponentially or faster at infinity 
(see \rref{;rtmain1} below). 
We note that this is an essential 
extension; since the proof given in \cite{bd} relies on 
several results relevant to the weak convergence of probability measures 
(its Lemma~2.8 for example), the boundedness of the functional $F$ 
seems inevitable. 
In this paper we use $L^{1}$-convergence results 
such as Scheff\'e's lemma instead, and show that the boundedness 
of $F$ is removable.  Our reasoning is applicable to the setting of 
an abstract Wiener space as well. 
Recently in \cite{ust}, \"Ust\"unel has extended the 
representation \eqref{;vr0} to a class of unbounded functionals 
to characterize in terms of the relative entropy the invertibility of 
path transformations of Brownian motion $\br $ of the form 
$\br +\int _{0}^{\cdot }\v _{s}\,ds$. Our proof of the extension differs 
from his and the condition we draw on the functionals is 
considerably weaker than that imposed in \cite{ust}; see 
\rsref{;rtmain1} and \ref{;rplb}. 

Pr\'ekopa's theorem states that, given a log-concave density function 
on a product of two finite-dimensional Euclidean spaces, say,  
$\R ^{m}\times \R ^{n}$, its $n$-dimensional marginal is also 
log-concave; this fact was originally proven by Pr\'ekopa \cite{pre} and then 
independently by Brascamp and Lieb \cite{bl} and Rinott \cite{rin}. 
As an application of the above-mentioned 
extension of \eqref{;vr0}, we prove an analogue of Pr\'ekopa's theorem 
for Wiener functionals. The derivation is straightforward once the extension 
of \eqref{;vr0} is established. This analogue of Pr\'ekopa's theorem is 
then applied to extend the so-called Brascamp-Lieb moment inequality 
\cite{bl} to the framework of Wiener spaces; 
our argument bypasses any discretization steps and 
it allows us to formulate the inequality in a fairly general 
situation in which no specific regularities such as continuity 
are required for functionals involved in it.  
We also refer to \rref{;rappend} in the appendix for 
another motivation for the extension of \eqref{;vr0}. 
We note that by employing a finite-dimensionalization procedure, 
Pr\'ekopa's theorem is extended to the setting of an abstract Wiener 
space in Feyel and \"Ust\"unel \cite{fu}; our framework for the 
extension is wider than that of \cite{fu} in some respect, which enables 
us to recover original Pr\'ekopa's theorem from ours under the 
integrability condition associated with the extension of  
\eqref{;vr0}. See \rref{;rtmain2}. 

We write $\W $ for the space $C([0,1];\R ^{d})$ of all 
$\R ^{d}$-valued continuous functions on $[0,1]$ vanishing 
at the origin, equipped with the norm 
\begin{align*}
 |w|_{\W }:=\sup _{0\le t\le 1}|w(t)|, \quad w\in \W . 
\end{align*}
We denote by $\mathcal{B}(\W )$ the associated Borel $\sigma $-field 
and by $\pr $ the Wiener measure on $(\W ,\mathcal{B}(\W ))$. In the sequel 
we denote by $\br $ the 
coordinate mapping process on $\W $: 
\begin{align*}
 \br _{t}(w):=w(t), \quad 0\le t\le 1,\, w\in \W . 
\end{align*}
We set 
\begin{align*}
 \calF _{t}:=\sigma (\br _{s},0\le s\le t)\vee \mathcal{N}, \quad 
 0\le t\le 1, 
\end{align*}
the filtration generated by $\br $ and augmented by the set 
$\mathcal{N}$ of all $\pr $-null events. We denote by 
$\A $ the set of all $\R ^{d}$-valued $\{ \calF _{t}\} $-progressively 
measurable processes 
$
\v 
=\bigl\{ \v _{t}=\bigl( \v ^{(1)}_{t},\ldots ,\v ^{(d)}_{t}\bigr) \bigr\} 
_{0\le t\le 1}
$ 
satisfying 
\begin{align*}
 \int _{0}^{1}\ex \left[ |\v _{t}|^{2}\right] dt<\infty . 
\end{align*}
Here and in what follows, $\ex $ denotes the expectation with 
respect to $\pr $ and $|x|$ stands for the Euclidean norm of 
$x\in \R ^{d}$. 

Let $F:\W \to \R $ be measurable. We assume: 
\begin{enumerate}[($\mathrm{A}$1)]{}
 \item it holds that $\ex \left[ e^{F(\br )}\right] <\infty $; 
 
 \item there exists $\delta >0$ such that 
 \begin{align*}
  \ex \left[ 
  F_{-}(\br )^{1+\delta }
  \right] <\infty , 
 \end{align*}
 where we set 
 $F_{-}(w):=\max \{ -F(w),0\} ,\,w\in \W $. 
\end{enumerate}
One of the main results of the paper is then stated as follows: 

\begin{thm}\label{;tmain1}
 For any measurable function $F:\W \to \R $ satisfying 
 \thetag{A1} and \thetag{A2}, 
 the following variational representation holds: 
 \begin{align}\label{;vr1}
  \log \ex \left[ 
 e^{F(W)}
 \right] 
 =\sup _{\v \in \A }\ex \left[ 
 F\left( \br +\int _{0}^{\cdot }\v _{s}\,ds\right) 
 -\frac{1}{2}\int _{0}^{1}|\v _{s}|^{2}\,ds
 \right] . 
 \end{align}
\end{thm}

We denote by $\H $ the Cameron-Martin subspace of $\W $, namely 
$\H $ consists of all elements $h=(h_{1},\ldots ,h_{d})$ in $\W $ 
such that for each $i=1,\ldots ,d$, the coordinate $h_{i}$ is 
an absolutely continuous function whose derivative satisfies 
\begin{align*}
 \int _{0}^{1}\left( \dot{h}_{i}(t)\right) ^{2}dt<\infty ; 
\end{align*}
recall that $\H $ is a Hilbert space with respect to the inner 
product 
\begin{align*}
 \inner{h^{1}}{h^{2}}
 :=\sum _{i=1}^{d}\int _{0}^{1}
 \dot{h}_{i}^{1}(t)\dot{h}_{i}^{2}(t)\,dt, 
 \quad h^{1},h^{2}\in \H . 
\end{align*}
For every $h\in \H $, we denote $|h|_{\H }=\sqrt{\inner{h}{h}}$. 
The next theorem gives an analogue of Pr\'ekopa's 
theorem on the (classical) Wiener space $\left( \W ,\H ,\pr \right) $. 

\begin{thm}\label{;tmain2}
 Let $L$ be a real vector space and $\Lambda $ a convex subset of 
 $L$. We suppose $G:\W \times \Lambda \to \R $ to be such that: 
 \begin{itemize}
  \item[$\mathrm{(B1)}$] for each $\la \in \Lambda $, the mapping 
  $G(\cdot ,\la ):\W \to \R $ is measurable and satisfies 
  \thetag{A2}; 
  
  \item[$\mathrm{(B2)}$] it holds that 
  for any $w_{1},w_{2}\in \W $ with $w_{1}-w_{2}\in \H $, and 
  for any $\la _{1},\la _{2}\in \Lambda $ and $\theta \in [0,1]$, 
  \begin{align*}
   &G\left( 
   \theta w_{1}+(1-\theta )w_{2},\theta \la _1+(1-\theta )\la _{2}
   \right) \\
   &\ge \theta G(w_{1},\la _{1})+(1-\theta )G(w_{2},\la _{2})
   -\frac{1}{2}\theta (1-\theta )\left| 
   w_{1}-w_{2}
   \right| _{\H }^{2}. 
  \end{align*}
 \end{itemize}
 Then the mapping 
 $\Lambda \ni \la \mapsto \log \ex \left[ e^{G(\br ,\la )}\right] $
 is concave. Here we use the convention that $\log \infty =\infty $. 
\end{thm}

We give remarks on \tsref{;tmain1} and \ref{;tmain2}. 

\begin{rem}\label{;rtmain1}
\thetag{1}~Suppose that there exist constants $0\le C_{1}<1/2$, $0<\al <2$ and 
$C_{2}\ge 0$ such that for $\pr $-a.e.\ $w\in \W $,  
\begin{align*}
  \log \left( 1+F_{-}(w)\right) \le C_{2}\left( 1+|w|_{\W }^{\al }\right) 
  +C_{1}|w|_{\W }^{2}. 
 \end{align*}
Then the assumption \thetag{A2} is fulfilled, which 
may be deduced from the fact (see, e.g., 
\cite[Exercise~4.4.13]{ks}) that for all $a<1/2$, 
\begin{align*}
 \ex \left[ 
 \exp \left( a|\br |_{\W }^2\right) 
 \right] <\infty . 
\end{align*}

\noindent 
\thetag{2}~Under the assumption 
\thetag{A1},  the right-hand side of \eqref{;vr1} is 
well-defined in the sense that for any $\v \in \A $, 
\begin{align*}
 \ex \left[ 
 F_{+}\left( \br +\int _{0}^{\cdot }\v _{s}\,ds\right) 
 \right] <\infty , \quad  F_{+}:=\max \left\{ F,0\right\} , 
\end{align*}
while 
$
 \ex \left[ 
 F_{-}\left( \br +\int _{0}^{\cdot }\v _{s}\,ds\right) 
 \right] 
$ 
may take the value $\infty $ for some $\v \in \A $; 
see the proof of \pref{;plb}. 
As will also be seen below, 
the supremum over $\v \in \A $ in the representation \eqref{;vr1} 
can be replaced by that over all $\v $'s in $\si $, a particular 
class of simple processes defined after \lref{;ebound}. This 
replacement allows us to remove the assumption \thetag{A1} as to 
the well-definedness mentioned above because 
we have 
$
 \ex \left[ 
 F_{-}\left( \br +\int _{0}^{\cdot }\v _{s}\,ds\right) 
 \right] <\infty 
$ 
for all $\v \in \si $; see \pref{;lu}. 

\noindent 
\thetag{3}~As shown in \cite[Section~5]{bd}, 
the representation \eqref{;vr0} for any bounded $F$ can be 
extended to any $F$ which is only assumed to be bounded from below. 
This extension is a direct consequence of the monotone convergence theorem: 
For each positive real $M$, truncating $F$ from above by 
$M$,  we have from \eqref{;vr0}, 
\begin{align*}
 \log \ex \left[ 
 e^{F_{M}(W)}
 \right] 
 =\sup _{\v \in \A }\ex \left[ 
 F_{M}\left( \br +\int _{0}^{\cdot }\v _{s}\,ds\right) 
 -\frac{1}{2}\int _{0}^{1}|\v _{s}|^{2}\,ds
 \right] , 
\end{align*}
where we set 
$
F_{M}=
\min \left\{ F,M\right\} . 
$ 
Then by the monotone convergence theorem, 
the left-hand side converges as $M\to \infty $ to 
the expression with $F_{M}$ replaced by $F$, and so does 
the right-hand side since 
\begin{equation}\label{;supsup}
 \begin{split}
 &\sup _{M>0}\sup _{\v \in \A }
 \ex \left[ 
 F_{M}\left( \br +\int _{0}^{\cdot }\v _{s}\,ds\right) 
 -\frac{1}{2}\int _{0}^{1}|\v _{s}|^{2}\,ds
 \right] 
 \\
 =&\sup _{\v \in \A }\sup _{M>0}
 \ex \left[ 
 F_{M}\left( \br +\int _{0}^{\cdot }\v _{s}\,ds\right) 
 -\frac{1}{2}\int _{0}^{1}|\v _{s}|^{2}\,ds
 \right] \\
 =&\sup _{\v \in \A }\ex \left[ 
 F\left( \br +\int _{0}^{\cdot }\v _{s}\,ds\right) 
 -\frac{1}{2}\int _{0}^{1}|\v _{s}|^{2}\,ds
 \right] . 
 \end{split}
\end{equation}
Therefore the essential part of \tref{;tmain1} is the 
removal of the boundedness from below of $F$. 

\noindent 
\thetag{4}~In \cite[Theorem~7]{ust}, the representation \eqref{;vr1} 
is proven to be valid under the condition that for some 
$p,q>1$ with $p^{-1}+q^{-1}=1$, 
\begin{align*}
 \ex \left[ 
 \left|F(\br )\right| ^{p}\right] <\infty \quad \text{and} \quad 
 \ex \left[ e^{qF(\br )}\right] <\infty 
\end{align*}
while our assumption of \tref{;tmain1} is equivalently rephrased as 
\thetag{A1} and 
$
 \ex \left[ 
 \left|F(\br )\right| ^{p}\right] <\infty 
$ for some $p>1$.  
\end{rem}

\begin{rem}\label{;rtmain2} 
\thetag{1}~Though the proof of \tref{;tmain2} is easily done in 
its generality, the generalization to any real vector space 
$L$ is not essential since the concavity is an expression on 
a line segment 
$\theta \la _{1}+(1-\theta )\la _{2},\,0\le \theta \le 1$, 
for every fixed $\la _{1}, \la _{2}\in \Lambda $. 

\noindent 
\thetag{2}~Let $g:\R ^{d}\times \Lambda \to \R $ be such that 
\begin{align}\label{;lcf}
 \text{the mapping $\R ^{d}\times \Lambda \ni (x,\la )\mapsto 
 g(x,\la )-|x|^{2}/2$ is concave. }
\end{align}
Then the functional $G$ defined by 
$G(w,\la )=g\left( w(1),\la \right) ,\,(w,\la )\in \W \times \Lambda $, 
satisfies the condition \thetag{B2} of \tref{;tmain2}; indeed, letting 
$w_{1},w_{2}$, $\la _{1},\la _{2}$, and $\theta $ be as in \thetag{B2}, 
we have from \eqref{;lcf}, 
\begin{align*}
 &g\bigl( \theta w_{1}(1)+(1-\theta )w_{2}(1),
 \theta \la _{1}+(1-\theta )\la _{2}
 \bigr) 
 -\theta g(w_{1}(1),\la _{1})-(1-\theta )g(w_{2}(1),\la _{2}) \\
 &\ge -\frac{1}{2}\theta (1-\theta )\left| w_{1}(1)-w_{2}(1)\right| ^{2}\\
 &\ge -\frac{1}{2}\theta (1-\theta )\left| w_{1}-w_{2}\right| _{\H }^{2}, 
\end{align*}
where the second inequality follows from the fact that for any $h\in \H $, 
\begin{align*}
 |h(1)|=\left| \int _{0}^{1}\dot{h}(t)\,dt\right| \le |h|_{\H }. 
\end{align*}
Therefore Pr\'ekopa's theorem in finite dimension is recovered 
from \tref{;tmain2} when the corresponding $G$ defined 
above satisfies the integrability condition in \thetag{B1}. 

\noindent 
\thetag{3}~The condition \thetag{B2} only concerns a pair of paths 
$w_{1}, w_{2}\in \W $ whose difference is in $\H $, which we think 
reflects the fact that the structure of the Wiener space 
$(\W ,\mathcal{B}(\W ),\pr )$ is determined by its 
{\itshape skeleton} $\H $. Introduced by Feyel and \"Ust\"unel \cite{fu} 
is the notion of $\H $-convexity, which, roughly speaking, is 
an almost sure convexity in the direction of $\H $. 
In Theorem~4.1 of \cite{fu}, Pr\'ekopa's theorem is extended to the 
product of two abstract Wiener spaces, which we rephrase in our 
present setting as follows: 
if $G:\W \times \W \to \R $ is measurable and 
$\H \times \H $-concave, namely 
\begin{equation}\label{;ustc}
 \begin{split}
 &G\left( 
 w_{1}+\theta h^{1}+(1-\theta )k^{1}, w_{2}+\theta h^{2}+(1-\theta )k^{2}
 \right) \\
 &\ge \theta G\left( w_{1}+h^{1},w_{2}+h^{2}\right) 
 +(1-\theta )G\left( w_{1}+k^{1},w_{2}+k^{2}\right) 
 \end{split}
\end{equation}
for $\pr \times \pr $-a.e.\ $(w_{1},w_{2})\in \W \times \W $ for 
every $h^{i},k^{i}\in \H ,\,i=1,2$, and $\theta \in [0,1]$, then the 
mapping 
\begin{align*}
 \W \ni w_{2}\mapsto \log \int _{\W }e^{G(w_{1},w_{2})}\,\pr (dw_{1})
\end{align*}
admits a version which is measurable and concave on $\W $. 
This assertion is proven by using 
finite-dimensional Pr\'ekopa's theorem 
and the Fourier expansion of elements in $\W $ along a given complete 
orthogonal basis of $\H $. While the above condition is weaker 
than \thetag{B2} in the respect that it allows a negligible set 
on which the relation \eqref{;ustc} fails, it does not allow  
the presence of the additional term 
$-(1/2)\theta (1-\theta )|h^{1}-k^{1}|_{\H }^{2}$ as in \thetag{B2}; 
it seems difficult to draw such a term from a finite-dimensionalizing 
procedure as used in \cite{fu}. 

\end{rem}

The rest of the paper is organized as follows: 
\sref{;prfvr} is devoted to the proof of \tref{;tmain1}. 
The lower bound in the representation \eqref{;vr1} is 
proven in \ssref{;ssplb}; we prove the upper bound 
in \ssref{;sspub} using the key \pref{;pkey} whose proof is 
given in \ssref{;prfkey}; we also show in \ssref{;prfkey} a 
variant of \tref{;tmain1} as \pref{;lu}, which is deduced 
from the proof of the theorem. 
In \sref{;prekopa} we prove \tref{;tmain2} 
and provide its application to the extension of 
the Brascamp-Lieb inequality to the framework of 
Wiener spaces. The proof of \tref{;tmain2} is given 
in \ssref{;prfpre} by using  \pref{;lu}; in \ssref{;bl} we 
formulate and prove the Brascamp-Lieb inequality 
on the Wiener space by applying \tref{;tmain2}. 
In the appendix, we discuss an extension of the 
Brascamp-Lieb inequality to the framework of 
nonconvex potentials in the case of one dimension. 
\bigskip 

For every $a,b\in \R $, we write $a\vee b=\max\{ a,b\} ,\ 
a\wedge b=\min \{ a,b\} $. For every $x,y\in \R ^{d}$, 
we write $x\cdot y$ for the inner product of $x$ and $y$ 
in $\R ^{d}$ and denote $|x|=\sqrt{x\cdot x}$ as above. 
For every $1\le p\le \infty $, we denote by $L^{p}(\pr )$ 
the set of all $\R $-valued random variables $X$ defined on 
the probability space $(\W ,\mathcal{B}(\W ),\pr )$ such that 
\begin{align*}
 \left\{ \no{X}{p}\right\} ^{p}:=\ex \left[ |X|^{p}\right] <\infty &
 && \text{for }p<\infty 
\intertext{and }
 \no{X}{\infty }:=\esssup _{w\in \W }|X(w)|<\infty &
 && \text{for }p=\infty . 
\end{align*} 
Here and in what follows the notation 
$\esssup \limits_{w\in \W }$ stands for the essential supremum 
over $w\in \W $ with respect to $\pr $. Other notation will be 
introduced as needed.

\section{Proof of \tref{;tmain1}}\label{;prfvr} 

In this section we give a proof of \tref{;tmain1}. 

For each $\v \in \A $, we denote by $\tr{\v }{}$ the path 
transform defined by 
\begin{align*} 
 \tr{\v }{t}(w):=w(t)+\int _{0}^{t}\v _{s}(w)\,ds, 
 \quad 0\le t\le 1,\ w\in \W . 
\end{align*}
We also set the process 
$\E{\v }=\{ \E{\v }_{t}\} _{0\le t\le 1}$ to 
be an $\{ \calF \} _{t}$-local martingale defined by 
\begin{align*}
 \E{\v }_{t}:=\exp 
 \left( 
 \int _{0}^{t}\v _{s}\cdot dW_{s}-\frac{1}{2}\int _{0}^{t}|\v _{s}|^2\,ds
 \right) , \quad 0\le t\le 1. 
\end{align*}
In the case that $\E{\v }$ is a true martingale, we define the probability 
measure $\pr ^{\v }$ on $(\W ,\mathcal{B}(\W ))$ by 
\begin{align}\label{;dpv}
 \pr ^{\v }(A):=\ex \left[ 
 \ind _{A}\E{\v }_{1}
 \right] , \quad A\in \mathcal{B}(\W ), 
\end{align}
and denote by $\ex ^{\v }$ the expectation with respect to 
$\pr ^{\v }$. By Girsanov's formula, the process 
$\tr{-v}{}(\br )$ is a standard Brownian motion 
under $\pr ^{\v }$, which may be rephrased in the 
statement that the identity 
\begin{align}\label{;gir}
 \ex ^{\v }\left[ F\left( \tr{-\v }{}(\br )\right) \right] 
 =\ex \left[ F(\br )\right] 
\end{align}
holds for any nonnegative 
measurable functional $F$ on $\W $. 

We say that an element $\v $ in $\A $ is bounded if 
\begin{align*}
 \sup _{0\le t\le 1}\no{|\v _{t}|}{\infty }<\infty . 
\end{align*}
The set of all bounded elements in $\A $ will be denoted by $\A _{b}$. 
Well-known Novikov's condition implies that if 
$\v \in \A _{b}$, then $\E{\v }$ is a martingale. 
The following simple fact will also be referred to 
frequently: 

\begin{lem}\label{;ebound}
 Suppose that $\v \in \A _{b}$. Then it holds that for any 
 $p>1$ and $0\le t\le 1$, 
 \begin{align*}
 \ex \left[ 
 \left( \E{\v }_{t}\right) ^{p}
 \right] 
 \le \exp \left\{ 
 \frac{1}{2}p(p-1)\sup _{0\le t\le 1}\no{|\v _{t}|}{\infty }^2
 \right\} . 
 \end{align*}
\end{lem}
\begin{proof}
 By the definition of $\E{\v }$, we have 
 \begin{align*}
  \left( \E{\v }_{t}\right) ^{p}
  =\E{p\v }_{t}\exp \left\{ 
  \frac{1}{2}p(p-1)\int _{0}^{t}|\v _{s}|^{2}\,ds
  \right\} . 
 \end{align*}
 Since the process $\E{p\v }$ is also a martingale by the boundedness of 
 $\v $, we have $\ex \left[ \E{p\v }_{t}\right] =1$ for all $0\le t\le 1$, 
 from which the claimed estimate follows readily. 
\end{proof}

We denote by $\si $ the set of all $\R ^{d}$-valued 
processes given in the form 
\begin{align}\label{;simplep}
 \v _{t}(w)=\xi _{0}\ind _{[t_{0},t_{1}]}(t)
 +\sum _{k=1}^{m-1}\xi _{k}(w)\ind _{(t_{k},t_{k+1}]}(t), 
 \quad 0\le t\le 1,\ w\in \W , 
\end{align}
for some $m\in \N $, $0=t_{0}<t_{1}<\cdots <t_{m}=1$, 
$\xi _{0}\in \R ^{d}$, and $\R ^{d}$-valued 
bounded continuous 
functionals $\xi _{k}(w)=\xi _{k}(w(t),t\le t_{k}),\,w\in \W $,  
$k=1,\ldots ,m-1$. 
We may deduce from \cite[Lemma~II.1.1]{iw} 
that $\si $ is dense in $\A $ 
with respect to the metric $\nt{\cdot }{\A }$ defined by 
\begin{align*}
 \nt{\v }{\A }^{2}:=
 \ex \left[ 
 \int _{0}^{1}|\v _{s}|^{2}\,ds
 \right] , \quad \v \in \A ; 
\end{align*}
see also discussions in \cite[Lemma~3.2.4, Problem~3.2.5]{ks} 
as to the density of $\si $ in $\A _{b}$. 

\subsection{Proof of the lower bound}\label{;ssplb}
In this subsection we give a proof of the lower bound in 
\eqref{;vr1}, namely with the notation above, we prove 

\begin{prop}\label{;plb}
Assume that a measurable function $F:\W \to \R $ satisfies 
\thetag{A1}. Then it holds that 
\begin{align}\label{;vr1l}
 \log \ex \left[ 
 e^{F(W)}
 \right] 
 \ge \sup _{\v \in \A }\left\{ 
 \ex \left[ F\left( 
 \tr{\v }{}(\br )
 \right) 
 \right] -\frac{1}{2}\nt{\v }{\A }^{2}
 \right\} .
\end{align}
\end{prop}

\begin{rem}\label{;rplb}
 In \cite[Theorem~6]{ust}, the lower bound \eqref{;vr1l} is 
 proven under the condition that 
 $
 \left( 1+|F(\br )|\right) e^{F(\br )} \in L^{1}(\pr )
 $. 
\end{rem}

The proof of \pref{;plb} is immediate if we are given the following 
lemma. 

\begin{lem}\label{;llb}
 The lower bound \eqref{;vr1l} holds for any bounded and measurable $F$. 
\end{lem}

Using this lemma, we prove \pref{;plb}. 
\begin{proof}[Proof of \pref{;plb}]
 First we verify that under the assumption \thetag{A1}, 
 \begin{align}\label{;fp}
  \ex \left[ 
  F_{+}\left( \tr{\v }{}(\br )\right) 
  \right] <\infty \quad 
  \text{for any }\v \in \A , 
 \end{align}
 where $F_{+}(w):=F(w)\vee 0,\,w\in \W $. 
 Fix $\v \in \A $ arbitrarily and set  
 $F_{+,M}=F_{+}\wedge M$ for each $M>0$. 
 Then by \lref{;llb}, we have in particular 
 \begin{align*}
  \ex \left[ 
  F_{+,M}\left( \tr{\v }{}(\br )\right) 
  \right] \le \log \ex \left[ 
  e^{F_{+,M}(\br )}\right] +\frac{1}{2}\nt{\v }{\A }^{2}. 
 \end{align*}
 Letting $M\to \infty $, we apply the monotone convergence theorem 
 to both sides to get 
 \begin{align*}
  \ex \left[ 
  F_{+}\left( \tr{\v }{}(\br )\right) 
  \right] &\le \log \ex \left[ 
  e^{F_{+}(\br )}\right] +\frac{1}{2}\nt{\v }{\A }^{2}\\
  &\le \log \ex \left[ 
  1+e^{F_{}(\br )}\right] +\frac{1}{2}\nt{\v }{\A }^{2}, 
 \end{align*}
 which is finite by \thetag{A1}. 
 
 For every $M,N>0$, we now define 
 \begin{align*}
  F_{N}(w):=F(w)\vee (-N), \quad 
  F_{N,M}(w):=F_{N}(w)\wedge M \quad \text{for }w\in \W . 
 \end{align*}
 Then by \lref{;llb}, the lower bound \eqref{;vr1l} 
 holds for $F_{N,M}$. By letting $M\to \infty $, the monotone 
 convergence theorem yields 
 \begin{align}\label{;vr1M}
  \lhs{F_{N}}\ge 
  \sup _{\v \in \A }
  \left\{ 
  \ex \left[ 
  F_{N}\left( 
  \tr{\v }{}(\br )
  \right) 
  \right] 
  -\frac{1}{2}\nt{\v }{\A }^{2}
  \right\} 
 \end{align}
 (cf.~\rref{;rtmain1}\,\thetag{3}). By the assumption \thetag{A1}, 
 the random variable $\sup \limits_{N>0}e^{F_{N}(\br )}$ is integrable 
 and so is $\sup \limits_{N>0}F_{N}\left( \tr{\v }{}(\br )\right) $ 
 for any $\v \in \A $ thanks to \eqref{;fp}. 
 Therefore as $N\to \infty $, we may use the monotone 
 convergence theorem on both sides 
 of \eqref{;vr1M} to obtain 
 \begin{align*}
  \lhs{F}&\ge \inf _{N>0}\sup _{\v \in \A }
  \left\{ 
  \ex \left[ 
  F_{N}\left( 
  \tr{\v }{}(\br )
  \right) 
  \right] 
  -\frac{1}{2}\nt{\v }{\A }^{2}
  \right\} \\
  &\ge \sup _{\v \in \A }\inf _{N>0}
  \left\{ 
  \ex \left[ 
  F_{N}\left( 
  \tr{\v }{}(\br )
  \right) 
  \right] 
  -\frac{1}{2}\nt{\v }{\A }^{2}
  \right\} \\
  &=\sup _{\v \in \A }
  \left\{ 
  \ex \left[ 
  F_{}\left( 
  \tr{\v }{}(\br )
  \right) 
  \right] 
  -\frac{1}{2}\nt{\v }{\A }^{2}
  \right\} , 
 \end{align*}
 which shows the proposition. 
\end{proof}

The statement of \lref{;llb} is the same as what is proven 
in the first half of the proof of Theorem~3.1 in 
Bou\'e-Dupuis \cite{bd}. 
For the self-containedness of  the paper, we give a proof 
of the lemma, which slightly differs from and simplifies the 
original one. 

We begin with the next two lemmas, assertions of which are 
taken respectively from pages 1648 and 1649 of \cite{bd}. 

\begin{lem}\label{;llb1}
Let $F:\W \to \R $ be bounded and measurable. It holds that for any 
$\v \in \A _{b}$, 
\begin{align*}
 \lhs{F}\ge \ex ^{\v }\left[ 
 F(\br )-\frac{1}{2}\int _{0}^{1}|\v _{s}|^{2}\,ds
 \right] . 
\end{align*}
\end{lem}

\begin{proof}
It is readily seen that 
\begin{align*}
 \lhs{F}-\ex ^{\v }\left[ 
 F(\br )
 -\log \E{\v }_{1}\right] 
 &=\ex ^{\v }\left[ 
 \log \left( 
 \frac{\ex \left[ e^{F(\br )}\right] \E{\v }_{1}}{e^{F(\br )}}
 \right) 
 \right] \\
 &\ge \ex \left[ 
 \left( 
 1-\frac{e^{F(\br )}}{\ex \left[ e^{F(\br )}\right] \E{\v }_{1}}
 \right) \E{\v }_{1}
 \right] \\
 &=1-1=0. 
\end{align*}
Here for the second line we used the inequality $\log x\ge 1-1/x$ for all 
$x>0$, and the definition \eqref{;dpv} of $\pr ^{\v }$. 
The proof of the lemma ends by noting that 
\begin{align}
 \ex ^{\v }\left[ 
 \log \E{\v }_{1}
 \right] &=
 \ex ^{\v }\left[ 
 \int _{0}^{1}\v _{s}\cdot d\br _{s}-\int _{0}^{1}|\v _{s}|^{2}\,ds
 \right] 
 +\ex ^{\v }\left[ 
 \frac{1}{2}\int _{0}^{1}|\v _{s}|^{2}\,ds
 \right] \notag \\
 &=\ex ^{\v }\left[ 
 \frac{1}{2}\int _{0}^{1}|\v _{s}|^{2}\,ds
 \right] \label{;exlog}
\end{align}
because of the fact that the process 
\begin{align*}
 \left( 
 \int _{0}^{t}\v _{s}\cdot d\br _{s}-\int _{0}^{t}|\v _{s}|^{2}\,ds
 \right) \E{\v }_{t},\quad 0\le t\le 1, 
\end{align*}
is a martingale by It\^o's formula, the boundedness of 
$\v $ and \lref{;ebound}. 
\end{proof}

\begin{lem}\label{;llb2}
Let $F:\W \to \R $ be bounded and measurable. It holds that 
for any $\v \in \si $, 
\begin{align}\label{;simplev}
 \lhs{F}\ge 
 \ex \left[ 
 F(\tr{\v }{}(\br ))
 \right] -\frac{1}{2}\nt{\v }{\A }^{2}. 
\end{align}
\end{lem}

\ND\tv{\widetilde{\v}}
\ND\txi{\widetilde{\xi}}
\begin{proof}
 Let $\v \in \si $ is written as \eqref{;simplep}. We construct from 
 $\v $ a process $\tv $ in such a way that for each $w\in \W $,  
 \begin{align*}
  &\txi _{0}:=\xi _{0}, && 
  \tv _{t}(w):=\txi _{0} \ \text{for }t_{0}\le t\le t_{1}, \\
  &\txi _{1}(w):=\xi _{1}
  \left( 
  w(t)-\int _{0}^{t}\tv _{s}(w)\,ds,t\le t_{1}
  \right) , && \tv _{t}(w):=\txi _{1}(w) \ \text{for }t_{1}<t\le t_{2}, \\
  &\cdots && \cdots \\
  &\txi _{m-1}(w):=\xi _{m-1}
  \left( 
  w(t)-\int _{0}^{t}\tv _{s}(w)\,ds,t\le t_{m-1}
  \right) , && \tv _{t}(w):=\txi _{m-1}(w) \ \text{for }t_{m-1}<t\le t_{m}, 
 \end{align*}
 so that we have the relation 
 \begin{align}\label{;compos}
  \tv (w)=\v \left( \tr{-\tv }{}(w)\right) , 
  \quad \tr{\v }{}\circ \tr{-\tv }{}(w)=w
 \end{align}
 for all $w\in \W $. It is clear by construction that 
 $\tv $ is in $\si $, and hence in $\A _{b}$. 
 Therefore by Girsanov's formula \eqref{;gir}, 
 the right-hand side of \eqref{;simplev} is equal to 
 \begin{align*}
  &\ex ^{\tv }\left[ 
  F\left( 
  \tr{\v }{}\circ \tr{-\tv }{}(\br )
  \right) -\frac{1}{2}\int _{0}^{1}
  \left| 
  \v _{s}\left( \tr{-\tv }{}(\br )\right) 
  \right| ^{2}ds
  \right] \\
  &=\ex ^{\tv }\left[ 
  F(\br )-\frac{1}{2}\int _{0}^{1}\left| \tv _{s}\right| ^{2}ds
  \right] , 
 \end{align*}
 where we used \eqref{;compos} for the second line. By \lref{;llb1}, 
 the last expression is dominated by $\lhs{F}$. This ends the proof. 
\end{proof}

Using \lref{;llb2}, we prove 
\begin{lem}\label{;llb3}
Let $F:\W \to \R $ be bounded and continuous. Then \eqref{;simplev} 
holds for any $\v \in \A $. 
\end{lem}

\begin{proof}
 Let $\v \in \A $. By the density of $\si $ in $\A $, there exists 
 a sequence $\{ \v ^{n}\} _{n\in \N }\subset \si $ such that 
 $\nt{\v ^{n}-\v }{\A }\to 0$ as $n\to \infty $. Then, since 
 \begin{align*}
  \ex \left[ 
  \left| 
  \int _{0}^{\cdot }\v ^{n}_{s}\,ds-\int _{0}^{\cdot }\v _{s}\,ds
  \right| _{\W }^{2}
  \right] 
  \le \nt{\v ^{n}-\v }{\A }^{2}
  \xrightarrow[n\to \infty ]{}0, 
 \end{align*}
 we may extract a subsequence $\{ n'\} \subset \N $ such that 
 \begin{align}\label{;unicon}
  \left| 
  \int _{0}^{\cdot }\v ^{n'}_{s}\,ds-\int _{0}^{\cdot }\v _{s}\,ds
  \right| _{\W }\xrightarrow[n'\to \infty ]{}0 
  \quad \text{a.s. }
 \end{align}
 Since each $\v ^{n'}$ is in $\si $ and $F$ is assumed to be bounded, 
 we have by \lref{;llb2}, 
 \begin{align*}
  \lhs{F}\ge \ex \left[ 
  F\left( 
  \tr{\v ^{n'}}{}(\br )
  \right) 
  \right] -\frac{1}{2}\nt{\v ^{n'}}{\A }^{2}. 
 \end{align*}
 By \eqref{;unicon} and the continuity of $F$, the bounded convergence 
 theorem yields 
 \begin{align*}
  \lim _{n'\to \infty }\ex \left[ 
  F\left( 
  \tr{\v ^{n'}}{}(\br )
  \right) 
  \right] =\ex \left[ 
  F\left( 
  \tr{\v ^{}}{}(\br )
  \right) 
  \right] . 
 \end{align*}
 As $\v ^{n'}$ approximates $\v $ with respect to 
 $\nt{\cdot }{\A }$, it also holds that 
 $\nt{\v ^{n'}}{\A }\to \nt{\v }{\A }$ as $n'\to \infty $. 
 Combining these ends the proof. 
\end{proof}

\begin{rem}\label{;rweakconv}
 In fact, when $\nt{\v ^{n}-\v }{\A }\to 0$ as $n\to \infty $, 
 the whole sequence $\{ \tr{\v ^{n}}{}(\br )\} _{n\in \N }$ 
 converges weakly to $\tr{\v }{}(\br )$. 
\end{rem}

We stand ready to prove \lref{;llb}. 
\begin{proof}[Proof of \lref{;llb}]
 Let $F:\W \to \R $ be bounded and measurable. Then there exists a 
 sequence $\{ F_{n}\} _{n\in \N }$ of bounded and continuous functions 
 on $\W $ such that 
 \begin{align}\label{;approxf}
  \lim _{n\to \infty }F_{n}=F \quad \text{a.s.} \ \quad  \text{and} \ \quad 
  \sup _{n\in \N }\no{F_{n}}{\infty }\le \no{F}{\infty }
 \end{align}
 as is recalled in \cite[Theorem~2.6]{bd} from 
 \cite[Theorem~V.16\,\thetag{a}]{doob}. Take $\v \in \A $ arbitrarily. 
 Then by \lref{;llb3}, we have for every $n\in \N $, 
 \begin{align*}
  \lhs{F_{n}}\ge \ex \left[ 
  F_{n}\left( 
  \tr{\v ^{}}{}(\br )
  \right) 
  \right] -\frac{1}{2}\nt{\v ^{}}{\A }^{2}. 
 \end{align*}
 The left-hand side converges to 
 $\lhs{F}$ as $n\to \infty $ by \eqref{;approxf} and 
 the bounded convergence theorem. Moreover, we also have 
 \begin{align*}
  \lim _{n\to \infty }\ex \left[ 
  F_{n}\left( 
  \tr{\v ^{}}{}(\br )
  \right) 
  \right] =
  \ex \left[ 
  F\left( 
  \tr{\v ^{}}{}(\br )
  \right) 
  \right] 
 \end{align*}
 since the law $\pr \circ \left( \tr{\v }{}\right) ^{-1}$ 
 is absolutely continuous with respect to $\pr $ 
 (see \cite[Theorem~4]{ls1} and \cite[Theorem~7.4]{ls2}) 
 thanks to $\int _{0}^{1}|\v _{s}|^{2}\,ds<\infty $,  
 $\pr $-a.s. 
 Combining these leads to the conclusion. 
\end{proof}

\begin{rem}\label{;rrelent}
 The above-mentioned absolute continuity may also be inferred from 
 the finiteness of the relative entropy of 
 $\pr \circ \left( \tr{\v }{}\right) ^{-1}$ with respect to $\pr $, 
 shown in equation \thetag{12} of \cite{bd}. 
\end{rem}

\subsection{Proof of the upper bound}\label{;sspub} 
In this subsection we prove the upper bound in 
\eqref{;vr1}:  

\begin{prop}\label{;pub}
Assume that a measurable function $F:\W \to \R $ satisfies 
\thetag{A1} and \thetag{A2}. Then it holds that 
\begin{align}\label{;eqpub}
 \log \ex \left[ 
 e^{F(W)}
 \right] 
 \le \sup _{\v \in \A }\left\{ 
 \ex \left[ F\left( 
 \tr{\v }{}(\br )
 \right) 
 \right] -\frac{1}{2}\nt{\v }{\A }^{2}
 \right\} .
\end{align}
\end{prop}

Using the notion of filtration introduced by \"Ust\"unel and Zakai \cite{uz} 
on an abstract Wiener space, Zhang \cite{zha} extended the variational 
representation \eqref{;vr0} of Bou\'e-Dupuis for bounded Wiener 
functionals to the framework of abstract Wiener spaces 
as simplifying considerably the original proof of the upper 
bound by employing the Clark-Ocone formula. We also 
make use of the Clark-Ocone formula to prove \pref{;pub}. 
 
First we prove \eqref{;eqpub} in the case that 
$F$ satisfies \thetag{A2} and is bounded from above: 
\begin{align}\label{;defM}
 M\equiv M_{F}:=\esssup _{w\in \W }F(w)<\infty . 
\end{align}
We denote by $\calF C_{b}^{1}$ the set of all functionals on $\W $ of the form 
\begin{align}\label{;cylinder}
 f\left( w(t_{1}),\ldots ,w(t_{m})\right) ,\quad w\in \W ,  
\end{align}
for some $m\in \N $, $0\le t_{1}<\cdots <t_{m}\le 1$ and for 
some bounded $C^{1}$-function $f:(\R ^{d})^{m}\to \R $ whose partial 
derivatives are all bounded as well. Since $F$ is in $L^{1}(\pr )$ by 
the assumption \thetag{A2} and \eqref{;defM}, we may find a sequence 
$\{ F_{n}\} _{n\in \N }\subset \calF C_{b}^{1}$ such that 
\begin{align}\label{;l1conv}
 \lim _{n\to \infty }\ex 
 \left[ 
 \left| 
 F_{n}(\br )-F(\br )
 \right| 
 \right] =0. 
\end{align}
Truncating $F_{n}$ if necessary, we may moreover assume that 
\begin{align}\label{;unifbd}
 \sup _{n\in \N }\sup _{w\in \W }F_{n}(w)\le M. 
\end{align}
We fix such a sequence $\{ F_{n}\} _{n\in \N }$. 
The following lemma is immediate from the Clark-Ocone formula 
and It\^o's formula. 

\begin{lem}\label{;co}
 For each $n\in \N $, there exists $\v ^{n}\in \A _{b}$ such that 
 \begin{align}\label{;coid}
 \frac{\ex \left[ e^{F_{n}(\br )}\big| \calF _{t}\right] }
 {\ex \left[ e^{F_{n}(\br )}\right] }
 =\E{\v ^{n}}_{t}\quad \text{a.s.} 
 \end{align}
 for all $0\le t\le 1$. 
\end{lem}

In fact, if $F_{n}$ is written as \eqref{;cylinder}, then the claimed 
$\v ^{n}$ admits the expression 
\begin{align*}
 \v ^{n}_{t}=
 \sum _{k=1}^{m}\ind _{[0,t_{k}]}(t)
 \frac{\ex \left[
 e^{F_{n}(\br )}\nabla _{x^{k}}f(\br (t_{1}),\ldots ,\br (t_{m}))
 \big| \calF _{t}
 \right] }
 {\ex \left[ e^{F_{n}(\br )}\big| \calF _{t}\right] }
 \quad \text{a.s.} 
\end{align*}
for all $0\le t\le 1$. For the Clark-Ocone formula, we 
refer the reader to 
\cite[Appendix~E]{ksF} and \cite[Proposition~1.3.14]{nua}. 

\begin{lem}\label{;relent}
 Let $\{ \v ^{n}\} _{n\in \N }\subset \A _{b}$ be as given in \lref{;co}. 
 Then for each $n\in \N $, we have 
 \begin{align*}
  \lhs{F_{n}}=\ex ^{\v ^{n}}\left[ 
  F_{n}(\br )-\frac{1}{2}\int _{0}^{1}\left| 
  \v ^{n}_{s}
  \right| ^{2}ds
  \right] . 
 \end{align*}
\end{lem}

\begin{proof}
 When $t=1$ we rewrite \eqref{;coid} in such a way that 
 \begin{align*}
  \lhs{F_{n}}=F_{n}(\br )-\log \E{\v ^{n}}_{1}\quad 
  \text{$\pr $-a.s.} 
 \end{align*}
 Taking the $\pr ^{\v ^{n}}$-expectation on the right-hand side and recalling 
 the identity \eqref{;exlog}, we have the lemma. 
\end{proof}

Using this lemma, we divide the left-hand side of 
\eqref{;eqpub} into three parts  
\begin{align}\label{;decompI}
 \lhs{F}=I^{1}_{n}+I^{2}_{n}+I^{3}_{n}
\end{align}
for each $n\in \N $, where we set 
\begin{align*}
 I^{1}_{n}&=\lhs{F}-\lhs{F_{n}}, \\
 I^{2}_{n}&=\ex ^{\v _{n}}\left[ 
 F_{n}(\br )-F(\br )
 \right] , \\
 I^{3}_{n}&=\ex ^{\v _{n}}\left[ 
 F(\br )-\frac{1}{2}\int _{0}^{1}\left| 
 \v ^{n}_{s}
 \right| ^{2}ds
 \right] . 
\end{align*}
Note that this decomposition makes sense because 
H\"older's inequality yields  
\begin{align*}
 \ex \left[ 
 F_{-}(\br )\E{\v ^{n}}_{1}
 \right] 
 \le \no{F_{-}(\br )}{1+\delta }\no{\E{v^{n}}_{1}}{1+1/\delta }<\infty 
\end{align*}
by the assumption \thetag{A2} and \lref{;ebound}. 

\begin{lem}\label{;convI}
 We have 
 \begin{align}\label{;convI12}
  \lim _{n\to \infty }I^{i}_{n}=0, \quad i=1,2. 
 \end{align}
\end{lem}

\begin{proof}
 Since the function $\R \ni z\mapsto e^{z}$ is increasing and 
 convex, we have 
 $\left| e^{z_{1}}-e^{z_{2}}\right| \le e^{M}|z_{1}-z_{2}|$ 
 for any $z_{1},z_{2}\le M$, and hence by 
 \eqref{;defM} and \eqref{;unifbd}, 
 \begin{align*}
  \left| \ex \bigl[ e^{F_{}(\br )}\bigr] -\ex \bigl[ e^{F_{n}(\br )}\bigr] \right| 
  \le e^{M}\ex \left[ \left| F(\br )-F_{n}(\br )\right| \right] 
 \end{align*}
 for all $n\in \N $. This implies \eqref{;convI12} for $i=1$ 
 by \eqref{;l1conv}. 
 
 As for $I^{2}_{n}$, we fix an $\ve >0$. By \eqref{;l1conv}, for 
 all sufficiently large $n$, 
 \begin{align*}
  \ex \left[ 
  F_{n}(\br )
  \right] \ge 
  \ex \left[ 
  F_{}(\br )
  \right] -\ve , 
 \end{align*}
 hence by Jensen's inequality, 
 \begin{align*}
  \ex \left[ 
  e^{F_{n}(\br )}
  \right] \ge \exp \left( 
  \ex \left[ F(\br )\right] -\ve 
  \right) . 
 \end{align*}
 By this estimate, \lref{;co} and \eqref{;unifbd}, we have 
 \begin{align*}
  \E{\v ^{n}}_{1}&=\frac{e^{F_{n}(\br )}}
  {\ex \left[ e^{F_{n}(\br )}\right] }\\
  &\le \exp \left( M+\ve -\ex \left[ F(\br )\right] \right) 
 \end{align*}
 if $n$ is sufficiently large. Then by the definition of \eqref{;dpv} of 
 $\pr ^{\v ^{n}}$, 
 \begin{align*}
  \left| I^{2}_{n}\right| 
  \le \exp \left( M+\ve -\ex \left[ F(\br )\right] \right) 
  \ex 
 \left[ 
 \left| 
 F_{n}(\br )-F(\br )
 \right| 
 \right] , 
 \end{align*}
 which tends to 0 as $n\to \infty $ by \eqref{;l1conv}. The proof is 
 complete. 
\end{proof}

As for $I^{3}_{n}$, we have the estimate 
\begin{align}\label{;key}
 \sup _{n\in \N }I^{3}_{n}\le 
 \sup _{\v \in \A }\left\{ 
 \ex \left[ 
 F\left( \tr{\v }{}(\br )\right) 
 \right] -\frac{1}{2}\nt{\v }{\A }^{2}
 \right\} , 
\end{align}
proof of which is postponed to \ssref{;prfkey}. 
Putting \eqref{;decompI}, \eqref{;convI12} and \eqref{;key} together, 
we have now arrived at 
\begin{prop}\label{;pubba}
 The upper bound \eqref{;eqpub} holds for any measurable 
 function $F:\W \to \R $ that satisfies \thetag{A2} 
 and is bounded from above. 
\end{prop}

\begin{proof}
By \eqref{;decompI} and \eqref{;key}, we have 
\begin{align*}
 \log \ex \left[ 
 e^{F(\br )}
 \right] \le I^{1}_{n}+I^{2}_{n}+\sup _{\v \in \A }
 \left\{ \ex \left[ 
 F\left( \tr{\v }{}(\br )\right) 
 \right] -\frac{1}{2}\nt{\v }{\A }^{2}
 \right\} 
\end{align*}
for all $n\in \N $. By letting $n\to \infty $, 
the assertion follows from \lref{;convI}. 
\end{proof}

The proof of \pref{;pub} is immediate from \pref{;pubba}. 
\begin{proof}[Proof of \pref{;pub}]
 For a measurable function $F:\W \to \R $ satisfying \thetag{A1} and 
 \thetag{A2}, we set for each $N>0$, 
 \begin{align*}
  F_{N}(w):=F(w)\wedge N, \quad w\in \W . 
 \end{align*}
 Then for any $N$, we have by \pref{;pubba}, 
 \begin{align*}
  \lhs{F_{N}}&\le 
  \sup _{\v \in \A }\left\{ 
  \ex \left[ F_{N}\left( 
  \tr{\v }{}(\br )
  \right) 
  \right] -\frac{1}{2}\nt{\v }{\A }^{2}
  \right\} \\
  &\le \sup _{\v \in \A }\left\{ 
 \ex \left[ F\left( 
 \tr{\v }{}(\br )
 \right) 
 \right] -\frac{1}{2}\nt{\v }{\A }^{2}
 \right\} . 
 \end{align*}
Letting $N\to \infty $ on the leftmost side leads to the conclusion 
by the dominated convergence theorem. 
\end{proof}

\subsection{Proof of \eqref{;key}}\label{;prfkey}
\ND\bv{\overline{\v}}

In this subsection we prove the estimate \eqref{;key}, which 
we rephrase in a slightly stronger statement that 
\begin{prop}\label{;pkey}
 Let $F:\W \to \R $ be a measurable function satisfying \thetag{A2} 
 and \eqref{;defM}. Then it holds that for any $\v \in \A _{b}$, 
 \begin{align}\label{;eqkey}
 \ex ^{\v _{}}\left[ 
 F(\br )-\frac{1}{2}\int _{0}^{1}\left| 
 \v ^{}_{s}
 \right| ^{2}ds
 \right] \le 
 \sup _{\bv \in \si }
 \left\{ 
 \ex \left[ F\left( 
 \tr{\bv }{}(\br )
 \right) 
 \right] -\frac{1}{2}\nt{\bv }{\A }^{2}
 \right\} . 
 \end{align}
\end{prop}

A key to the proof of this proposition 
is \lref{;scheffe} below, which is an 
immediate consequence of Scheff\'e's lemma. 
We note that Scheff\'e's lemma is also employed 
in Osuka \cite{osu}, where the variational 
representation \eqref{;vr0} of Bou\'e-Dupuis 
is extended to bounded functionals of 
$G$-Brownian motion, an extended notion of 
Brownian motion introduced by Peng \cite{pen1,pen2}, 
to the framework of sublinear expectation spaces. 

Fix $\v \in \A _{b}$ and let $\{ \v ^{n}\} _{n\in \N }\subset \si $ 
be such that 
\begin{align}\label{;approx}
 \lim _{n\to \infty }\nt{\v ^{n}-\v }{\A }=0 
\end{align}
and that by truncating each $\v ^{n}$ if necessary, 
\begin{align}\label{;Kbd}
 \sup _{n\in \N }\sup _{0\le t\le 1}\no{|\v ^{n}_{t}|}{\infty }
 \le \sup _{0\le t\le 1}\no{|\v ^{}_{t}|}{\infty }=:K<\infty . 
\end{align}
Note that such an approximate sequence exists 
by the density of $\si $ in $\A $. 

\begin{lem}\label{;scheffe}
 It holds that 
 \begin{align}\label{;scheffe1}
  \lim _{n\to \infty }
  \no{\E{\v ^{n}}_{1}-\E{\v }_{1}}{1}=0. 
 \end{align}
\end{lem}

\begin{proof}
 Fix an arbitrary subsequence $\{ n'\} \subset \N $. It suffices to 
 prove the existence of a subsequence of $\{ n'\} $ 
 along which the convergence \eqref{;scheffe1} takes place. 
 By It\^o's isometry and \eqref{;approx}, 
 \begin{align*}
  \ex \left[ 
  \left| 
  \int _{0}^{1}\v ^{n'}_{s}\cdot d\br _{s}
  -\int _{0}^{1}\v ^{}_{s}\cdot d\br _{s}
  \right| ^2
  \right] =\nt{\v ^{n'}-\v }{\A }^{2}
  \xrightarrow[n'\to \infty ]{}0. 
 \end{align*}
 Moreover, we have by \eqref{;approx} and \eqref{;Kbd}, 
 \begin{align*}
  \ex \left[ 
  \left| 
  \int _{0}^{1}\bigl| \v ^{n'}_{s}\bigr| ^{2}ds
  -\int _{0}^{1}\left| \v ^{}_{s}\right| ^{2}ds
  \right| 
  \right] 
  \le 2K\ex \left[ 
  \int _{0}^{1}\bigl| 
  \v ^{n'}_{s}-\v _{s}
  \bigr| ds
  \right] 
  \xrightarrow[n'\to \infty ]{}0. 
 \end{align*}
 Therefore we may extract a subsequence 
 $\{ n''\} \subset \{ n'\} $ such that 
 \begin{align}\label{;asconv}
  \lim _{n''\to \infty }\left( \int _{0}^{1}\v ^{n''}_{s}\cdot d\br _{s}
  -\frac{1}{2}\int _{0}^{1}\bigl| \v ^{n''}_{s}\bigr| ^{2}ds\right) 
  =\int _{0}^{1}\v ^{}_{s}\cdot d\br _{s} 
  -\frac{1}{2}\int _{0}^{1}\bigl| \v ^{}_{s}\bigr| ^{2}ds 
  \quad \text{a.s.} 
 \end{align}
 By the boundedness of $v^{n''}$ and 
 $\v $, Novikov's condition entails that 
 \begin{align}\label{;const}
  \ex \bigl[ 
  \E{\v ^{n''}}_{1}
  \bigr] =\ex \left[ \E{\v }_{1}\right] =1 
  \quad \text{for all $n''$.}
 \end{align}
 By \eqref{;asconv}, \eqref{;const} and Scheff\'e's lemma, we 
 conclude that the convergence \eqref{;scheffe1} takes 
 place along $\{ n''\} $. This proves the lemma. 
\end{proof}

For every $n\in \N $, we decompose the left-hand side of \eqref{;eqkey} 
into the sum 
\begin{align}\label{;decompJ}
 \ex ^{\v _{}}\left[ 
 F(\br )-\frac{1}{2}\int _{0}^{1}\left| 
 \v ^{}_{s}
 \right| ^{2}ds
 \right] =J^{1}_{n}+\frac{1}{2}J^{2}_{n}+J^{3}_{n}, 
\end{align}
where we set 
\begin{align*}
 J^{1}_{n}
 &=\ex \left[ F(\br )\!\left( \E{\v }_{1}-\E{\v ^{n}}_{1}\right) \right] , \\
 J^{2}_{n}
 &=\ex ^{\v ^{n}}\left[ 
 \int _{0}^{1}\bigl| \v ^{n}_{s}\bigr| ^2ds
 \right] 
 -\ex ^{\v ^{}}\left[ 
 \int _{0}^{1}\bigl| \v ^{}_{s}\bigr| ^2ds
 \right] , \\
 J^{3}_{n}&=\ex ^{\v _{n}}\left[ 
 F(\br )-\frac{1}{2}\int _{0}^{1}\left| 
 \v ^{n}_{s}
 \right| ^{2}ds
 \right] . 
\end{align*}

\begin{lem}\label{;convJ}
 We have 
 \begin{align}\label{;convJ12}
  \lim _{n\to \infty }J^{i}_{n}=0, \quad i=1,2. 
 \end{align}
\end{lem}

\begin{proof}
 Fix $N>0$ arbitrarily. Observe the bound 
 \begin{align}
  \left| J^{1}_{n}\right| 
  &\le \ex \left[ 
  \bigl| F(\br )\ind _{\{ F(\br )>-N\} }\bigr| \bigl| \E{\v }_{1}-\E{\v ^{n}}_{1}\bigr| 
  \right] 
  +\ex \left[ 
  \bigl| F(\br )\ind _{\{ F(\br )\le -N\} }\bigr| 
  \bigl| \E{\v }_{1}-\E{\v ^{n}}_{1}\bigr| 
  \right] \notag \\
  &\le (M\vee N)\no{\E{\v }_{1}-\E{\v ^{n}}_{1}}{1}
  +C\no{F_{-}(\br )\ind _{\{ F_{-}(\br )\ge N\} }}{1+\delta} 
   \label{;estJ1}
 \end{align}
 with 
 $
  C:=\sup \limits_{n\in \N }\no{\E{\v }_{1}-\E{\v ^{n}}_{1}}{1+1/\delta }
 <\infty 
 $, 
 where for the second line we used \eqref{;defM} and H\"older's inequality; 
 the finiteness of $C$ is due to \eqref{;Kbd} and \lref{;ebound}. 
 Letting $n\to \infty $ on both sides of \eqref{;estJ1}, we see from 
 \lref{;scheffe} that 
 \begin{align*}
  \limsup _{n\to \infty }\left| J^{1}_{n}\right| 
  \le C\no{F_{-}(\br )\ind _{\{ F_{-}(\br )\ge N\} }}{1+\delta}
 \end{align*}
 for any $N>0$. Since the right-hand side tends to $0$ as $N\to \infty $ 
 by the assumption \thetag{A2}, we obtain \eqref{;convJ12} for $i=1$. 

 As for $J^{2}_{n}$, we observe that by \eqref{;Kbd}, 
 \begin{align*}
  \left| J^{2}_{n}\right| &\le 
  \ex \left[ 
  \int _{0}^{1}\bigl| \v ^{n}_{s}\bigr| ^2ds\,
  \bigl| \E{\v ^{n}}_{1}-\E{\v }_{1}\bigr| 
  \right] 
  +\ex \left[ 
  \left| 
  \int _{0}^{1}\left( \bigl| \v ^{n}_{s}\bigr| ^2-|\v _{s}|^2\right) ds
  \right| \E{\v }_{1}
  \right] \\
  &\le K^2\no{\E{\v ^{n}}_{1}-\E{\v }_{1}}{1}
  +2K\nt{\v ^{n}-\v }{\A }\no{\E{\v }_{1}}{2}. 
 \end{align*}
 Since $\no{\E{\v }_{1}}{2}<\infty $ by \lref{;ebound}, the last 
 expression tends to $0$ as $n\to \infty $ by \lref{;scheffe} and 
 \eqref{;approx}. The proof of the lemma is complete. 
\end{proof}

Concerning $J^{3}_{n}$, we have 
\begin{lem}\label{;J3}
 It holds that 
 \begin{align*}
 \sup _{n\in \N }J^{3}_{n}\le 
 \sup _{\bv \in \si }
 \left\{ 
 \ex \left[ F\left( 
 \tr{\bv }{}(\br )
 \right) 
 \right] -\frac{1}{2}\nt{\bv }{\A }^{2}
 \right\} . 
 \end{align*}
\end{lem}

\ND\bxi{\overline{\xi}}
\begin{proof}
 Fix $n\in \N $. Since $\v ^{n}$ is in $\si $, we may represent 
 $\v ^{n}$ as \eqref{;simplep}. We construct from $\v ^{n}$ 
 a process $\bv $ in such a way that for each $w\in \W $,  
 \begin{align*}
  &\bxi _{0}:=\xi _{0}, && 
  \bv _{t}(w):=\bxi _{0} \ \text{for }t_{0}\le t\le t_{1}, \\
  &\bxi _{1}(w):=\xi _{1}
  \left( 
  w(t)+\int _{0}^{t}\bv _{s}(w)\,ds,t\le t_{1}
  \right) , && \bv _{t}(w):=\bxi _{1}(w) \ \text{for }t_{1}<t\le t_{2}, \\
  &\cdots && \cdots \\
  &\bxi _{m-1}(w):=\xi _{m-1}
  \left( 
  w(t)+\int _{0}^{t}\bv _{s}(w)\,ds,t\le t_{m-1}
  \right) , && \bv _{t}(w):=\bxi _{m-1}(w) \ \text{for }t_{m-1}<t\le t_{m}. 
 \end{align*}
 Note that $\bv $ is in $\si $ by construction; moreover, 
 by induction on $k=1,\ldots ,m$, we have 
 for all $w\in \W $, 
 \begin{align}\label{;trrel1}
  \v ^{n}_{t}(w)=\bv _{t}\left( \tr{-\v ^{n}}{}(w)\right) , 
  \quad 0\le t\le t_{k}, \ k=1,\ldots ,m, 
 \end{align}
 from which it also follows that 
 \begin{align}\label{;trrel2}
  \tr{\bv }{}\circ \tr{-\v ^{n}}{}(w)=w \quad \text{for all }w\in \W . 
 \end{align}
 These relations were noticed in Zhang \cite{zha}. 
 Using \eqref{;trrel1} and \eqref{;trrel2}, 
 we rewrite $J^{3}_{n}$ as 
 \begin{align*}
  J^{3}_{n}&=\ex ^{\v ^{n}}
  \left[ 
  F\left( \tr{\bv }{}\circ \tr{-\v ^{n}}{}(\br )\right) 
  -\frac{1}{2}\int _{0}^{1}
  \left| \bv _{s}\left( \tr{-\v ^{n}}{}(\br )\right) \right| ^2ds
  \right] \\
  &=\ex \left[ 
  F\left( \tr{\bv }{}(\br )\right) 
  -\frac{1}{2}\int _{0}^{1}\left| \bv _{s}\right| ^2ds
  \right] , 
 \end{align*}
 where the second line follows from the boundedness 
 from above of $F$ and Girsanov's formula \eqref{;gir}. 
 Since $\bv \in \si $, the lemma is proven. 
\end{proof}

We are in a position to prove \pref{;pkey}. 
\begin{proof}[Proof of \pref{;pkey}]
 By \eqref{;decompJ} and \lref{;J3}, we have 
 \begin{align*}
 \ex ^{\v _{}}\left[ 
 F(\br )-\frac{1}{2}\int _{0}^{1}\left| 
 \v ^{}_{s}
 \right| ^{2}ds
 \right] \le J^{1}_{n}+\frac{1}{2}J^{2}_{n}+
 \sup _{\bv \in \si }
 \left\{ 
 \ex \left[ F\left( 
 \tr{\bv }{}(\br )
 \right) 
 \right] -\frac{1}{2}\nt{\bv }{\A }^{2}
 \right\} 
 \end{align*}
 for all $n\in \N $. The assertion 
 follows by letting $n\to \infty $ thanks to \lref{;convJ}. 
\end{proof}

\pref{;pkey} reveals that we may replace the supremum over 
$\v \in \A $ in the variational representation \eqref{;vr1} 
by that over $\v \in \si $; by adopting the convention that 
$\log \infty =\infty $, the representation \eqref{;vr1} with this 
replacement remains true even if we remove the assumption 
\thetag{A1}. For later use, we state it in a proposition. 

\begin{prop}\label{;lu}
 Let $F:\W \to \R $ be measurable and satisfy 
 \thetag{A2}. Then it holds that 
 \begin{align}\label{;vr2}
 \log \ex \left[ 
 e^{F(W)}
 \right] 
 =\sup _{\v \in \si }\left\{ 
 \ex \left[ F\left( 
 \tr{\v }{}(\br )
 \right) 
 \right] -\frac{1}{2}\nt{\v }{\A }^{2}
 \right\} , 
 \end{align}
 where the left-hand side is understood to be equal to 
 $\infty $ when $\ex \left[ e^{F(W)}\right] =\infty $. 
\end{prop}

\begin{proof}
 For every $M>0$, we set $F_{M}(w):=F(w)\wedge M,\,w\in \W $. 
 By \psref{;plb} and \ref{;pkey} together with the proof of 
 \pref{;pubba}, we see that \eqref{;vr2} holds 
 for $F_{M}$: 
 \begin{align}\label{;vr2M}
 \log \ex \left[ 
 e^{F_{M}(W)}
 \right] 
 =\sup _{\v \in \si }\left\{ 
 \ex \left[ F_{M}\left( 
 \tr{\v }{}(\br )
 \right) 
 \right] -\frac{1}{2}\nt{\v }{\A }^{2}
 \right\} . 
 \end{align}
 Letting $M\to \infty $, we have the convergence of the left-hand side 
 to 
 $
 \log \ex \left[ 
 e^{F_{}(W)}
 \right] 
 $ 
 by the monotone convergence theorem. As for the right-hand side, 
 note that 
 \begin{align*}
  \ex \left[ F_{-}\left( \tr{\v }{}(\br )\right) \right] <\infty 
  \quad \text{for each $\v \in \si $;} 
 \end{align*} 
 indeed, constructing from $\v $ a process 
 $\tv $ in $\si $ that satisfies the relation \eqref{;compos}, we have 
 \begin{align*}
  \ex \left[ 
  F_{-}\left( 
  \tr{\v }{}(\br )
  \right) 
  \right] &=\ex ^{\tv }\left[ 
  F_{-}\left( \tr{\v }{}\circ \tr{-\tv }{}(\br )\right) 
  \right] \\
  &=\ex \left[ 
  F_{-}(\br )\E{\tv }_{1}
  \right] \\
  &\le \no{F_{-}(\br )}{1+\delta }\no{\E{\tv }_{1}}{1+1/\delta }, 
 \end{align*}
 which is finite by \thetag{A2}, the boundedness of $\tv $ and 
 \lref{;ebound}. Here we used Girsanov's formula \eqref{;gir} 
 for the first line, the relation \eqref{;compos} and the definition 
 \eqref{;dpv} of $\pr ^{\tv }$ for the second, and H\"older's 
 inequality for the third. Therefore we may also apply the monotone 
 convergence theorem to the right-hand side of \eqref{;vr2M} to obtain 
 \begin{align*}
  \sup _{M>0}\sup _{\v \in \si }\left\{ 
 \ex \left[ F_{M}\left( 
 \tr{\v }{}(\br )
 \right) 
 \right] -\frac{1}{2}\nt{\v }{\A }^{2}
 \right\} 
 =\sup _{\v \in \si }\left\{ 
 \ex \left[ F_{}\left( 
 \tr{\v }{}(\br )
 \right) 
 \right] -\frac{1}{2}\nt{\v }{\A }^{2}
 \right\} 
 \end{align*}
 (cf.~\eqref{;supsup}), which concludes the proof. 
\end{proof}

We end this section with a remark on the proof of 
\tref{;tmain1}. 
\begin{rem}\label{;rptmain1}
\thetag{1}~The above proof is also valid in the setting of 
an abstract Wiener space;  we may extend the variational 
representation of Zhang \cite{zha} for bounded functionals on the 
abstract Wiener space to functionals satisfying conditions 
corresponding to \thetag{A1} and \thetag{A2}. 

\noindent 
\thetag{2}~Since both sides of \eqref{;vr1} are well-defined
only under the assumption \thetag{A1} as noted in 
\rref{;rtmain1}\,\thetag{2}, it seems plausible that the 
representation \eqref{;vr1} holds true without any 
assumptions on $F$ from below; however, 
we have not succeeded in proving it. The problem is 
how to prove the upper bound \eqref{;eqpub} without 
the integrability assumption \thetag{A2}. 
\end{rem}

\section{Pr\'ekopa's theorem on Wiener space and its application}
\label{;prekopa}

In this section we prove \tref{;tmain2} and provide its 
application in \tref{;tbl}, which extends the Brascamp-Lieb 
inequality to the Wiener space. 

\subsection{Proof of \tref{;tmain2}}\label{;prfpre}
In this subsection we give a proof of \tref{;tmain2} 
as an immediate application of \pref{;lu}.  

\begin{proof}[Proof of \tref{;tmain2}]
 Set $g(\la )=\log \ex \left[ e^{G(\br ,\la )}\right] ,\,\la \in \Lambda $. 
 We fix $\la _{1},\la _{2}\in \Lambda $ and $\theta \in [0,1]$ arbitrarily. 
 For any $\v ^{1},\v ^{2}\in \si $, 
 we have by the condition \thetag{B2}, 
 \begin{align*}
  &G\bigl( 
  \tr{\theta \v ^{1}+(1-\theta )\v ^{2}}{}(w), 
  \theta \la _{1}+(1-\theta )\la _{2}
  \bigr) 
  -\frac{1}{2}\int _{0}^{1}\left| 
  \theta \v ^{1}_{s}(w)+(1-\theta )\v ^{2}_{s}(w)
  \right| ^{2}ds\\
  &\ge \theta 
  \Bigl\{ 
  G( 
  \tr{\v ^{1}}{}(w), 
  \la _{1}
  ) -\frac{1}{2}\int _{0}^{1}\left| \v ^{1}_{s}(w)\right| ^{2}ds
  \Bigr\} 
  +(1-\theta )
  \Bigl\{ 
  G( 
  \tr{\v ^{2}}{}(w), 
  \la _{2}
  ) -\frac{1}{2}\int _{0}^{1}\left| \v ^{2}_{s}(w)\right| ^{2}ds
  \Bigr\} 
 \end{align*}
 for all $w\in \W $. 
 Noting $\theta \v ^{1}+(1-\theta )\v ^{2}\in \si $, we take the 
 expectation in $w$ with respect to $\pr $ on both sides to get  
 \begin{align*}
  g\left( \theta \la _{1}+(1-\theta )\la _{2}\right) 
  &=\sup _{\v \in \si }\left\{ 
  \ex \left[ 
 G\left( 
 \tr{\v }{}(\br ), 
 \theta \la _{1}+(1-\theta )\la _{2}
 \right) \right] 
 -\frac{1}{2}\nt{\v }{\A }^{2}
 \right\} \\
 &\ge \theta \left\{ \ex \left[ 
 G\bigl( 
 \tr{\v ^{1}}{}(\br ), 
 \la _{1}
 \bigr) \right] -\frac{1}{2}\nt{\v ^{1}}{\A }^{2}
 \right\} 
 \\&\hspace{30mm}
 +(1-\theta )
 \left\{ \ex \left[ 
 G\bigl( 
 \tr{\v ^{2}}{}(\br ),  
  \la _{2}
  \bigr) \right] -\frac{1}{2}\nt{\v ^{2}}{\A }^{2}
 \right\}  
 \end{align*}
 for any $\v ^{1},\v ^{2}\in \si $. Here the equality is due 
 to \thetag{B1} and \pref{;lu}. Maximizing the rightmost side over 
 $\v ^{1}$ and $\v ^{2}$, and using \pref{;lu}, we obtain 
 \begin{align*}
  g\left( \theta \la _{1}+(1-\theta )\la _{2}\right) 
  \ge \theta g(\la _{1})+(1-\theta )g(\la _{2}) 
 \end{align*}
 as claimed. 
\end{proof} 

\begin{rem}\label{;r2tmain2}
\thetag{1}~In \cite[Subsection~13.A]{sim}, 
some convexity results are shown as to the Schr\"odinger 
operator $-(1/2)\Delta +V$ in $\R ^{d}$ with $V$ a 
convex function, such as the log-concavity of 
its ground state and the convexity of the infimum of its 
spectrum relative to an additional parameter put into 
the operator; these are derived by employing the time 
discretization of the associated Feynman-Kac path integral 
representations and finite-dimensional Pr\'ekopa's theorem. 
We can also prove those results by using \tref{;tmain2}; 
the advantage 
is that discretization procedures are not required at all.  

\noindent 
\thetag{2}~\tref{;tmain2} can also be extended to the framework of abstract 
Wiener spaces. 
\end{rem}

\subsection{Brascamp-Lieb inequality on Wiener space}\label{;bl}
In this subsection, the Brascamp-Lieb inequality formulated 
on the Wiener space is shown as an application of \tref{;tmain2}. 

We denote by $\W ^{*}$ (resp.~$\H ^{*}$) the topological dual 
space of $\W $ (resp.~of $\H $) and by  
$
\cpl{\cdot }{\cdot }\equiv 
{}_{\W ^{*}}\!\cpl{\cdot }{\cdot } _{\W }
$ 
the natural coupling 
between $\W ^{*}$ and $\W $. 
Identifying $\H ^{*}$ with $\H $ and noting the inclusion 
$\W ^{*}\subset \H ^{*}$, we regard each $l\in \W ^{*}$ as an element 
in $\H $, which we still denote by $l$. 

\begin{thm}\label{;tbl}
 Let $F:\W \to \R $ be a measurable function satisfying the following 
 assumptions \thetag{C1}--\thetag{C3}: 
 \begin{itemize}
 \item[\thetag{C1}] $F$ is concave on $\W $; 
 
 \item[\thetag{C2}] $\ex \left[ e^{F(\br )}\right] <\infty $; 
 
 \item[\thetag{C3}] there exists $\delta >0$ such that 
 $F_{-}\in L^{1+\delta }(\pr )$. 
 \end{itemize}
 We define the probability measure $\Q $ on $(\W ,\mathcal{B}(\W ))$ by 
 \begin{align*}
  \Q (A):=\frac{\ex \left[ \ind _{A}e^{F(\br )}\right] }
  {\ex \left[ e^{F(\br )}\right] }, \quad A\in \mathcal{B}(\W ), 
 \end{align*}
 and denote by $\ex _{\Q }$ the expectation with respect to $\Q $. Then 
 it holds that for any nonzero $l\in \W ^{*}$ and for any 
 convex function $\psi $ on $\R $, 
 \begin{align}\label{;blw}
  \ex _{\Q }\left[ 
  \psi \bigl( 
  \cpl{l}{\br }
  -\ex _{\Q }\left[ \cpl{l}{\br }\right] 
  \bigr) 
  \right] 
  \le \frac{1}{\sqrt{2\pi }|l|_{\H }}
  \int _{\R ^{}}\psi (z)\exp \left( 
  -\frac{z^{2}}{2|l|_{\H }^{2}}
  \right) dz. 
 \end{align}
\end{thm}

\begin{rem}
 Suppose that $F:\W \to \R $ satisfies \thetag{C1} and 
 is upper-semicontinuous on $\W $. Then the assumption 
 \thetag{C2} is fulfilled because $F$ is bounded 
 from above by an affine function \cite[Proposition~2.20]{bp}; 
 we also refer to the fact that the concavity and upper-semicontinuity 
 of $F$ yield the continuity of $F$ \cite[Proposition~2.16]{bp} as $\W $ 
 is a Banach space. 
\end{rem}

Let $V:\R ^{d}\to \R $ be a convex function and $\Sigma $ a 
symmetric, positive definite $d\times d$-matrix. We consider the 
case that $F$ is given by 
\begin{align*}
 F(w)&=-V\left( \Sigma ^{1/2}w(1)\right) ,\quad w\in \W , 
\intertext{and satisfies \thetag{C3}, and that $l$ is of the form }
 \cpl{l}{w}&=\Sigma ^{1/2}\al \cdot w(1), \quad w\in \W , 
\end{align*}
for a given $\al \in \R ^{d}$ $(\al \neq 0)$. In this case the inequality 
\eqref{;blw} is restated as 
\begin{align}\label{;bl0}
 E \left[ 
 \psi \left( \al \cdot X-E\left[ \al \cdot X\right] \right) 
 \right] 
 \le  
 E \left[ 
 \psi \left( \al \cdot Y \right) 
 \right] 
\end{align}
for any convex function $\psi $ on $\R $. Here $X$ and $Y$ are 
$\R ^{d}$-valued random variables defined on a probability 
space $(\Omega ,\calF ,P)$, whose laws induced on $\R ^{d}$ are given 
respectively by 
\begin{align}\label{;laws}
 P\left( X\in dx\right) 
 =\frac{1}{Z}e^{-V(x)}\nu (dx), \quad P\left( Y\in dx\right) =\nu (dx) , 
\end{align}
where $\nu $ is the normal distribution with mean $0$ and 
covariance matrix $\Sigma $ and $Z$ is the normalizing constant. 
The inequality \eqref{;bl0} is referred to as the Brascamp-Lieb 
(moment) inequality; it was originally proven by Brascamp and Lieb 
\cite[Theorem~5.1]{bl} in the case $\psi (z)=|z|^{p},\,p\ge 1$, 
and later extended by Caffarelli \cite[Corollary~6]{caf} to general convex 
$\psi $'s based on analyses of the optimal transport between 
the laws of $X$ and $Y$. In \cite{har}, the author gives a proof of 
the Brascamp-Lieb inequality \eqref{;bl0} based on the Skorokhod embedding 
and the It\^o-Tanaka formula, and derives error estimates for the inequality 
in terms of the variances of $\al \cdot X$ and $\al \cdot Y$ 
\cite[Theorem~1.1]{har}. In these three papers proofs of \eqref{;bl0} are 
reduced to the one-dimensional case thanks to finite-dimensional 
Pr\'ekopa's theorem. The proof of \tref{;tbl} is done in 
the same way by employing \tref{;tmain2}, 
an infinite-dimensional version of Pr\'ekopa's theorem. 

\begin{proof}[Proof of \tref{;tbl}]
 Fix $l\in \W ^{*}$ $(l\neq 0)$. We may assume without loss of generality 
 that $|l|_{\H }=1$. Since the law of 
 $\cpl{l}{\br }$ under $\Q $ is expressed as 
 \begin{align*}
  \Q \left( 
  \cpl{l}{\br }\in dz
  \right) 
  =\frac{1}{\sqrt{2\pi }\ex \left[ e^{F(\br )}\right] }
  \exp \left( -\frac{z^2}{2}\right) 
  \ex \left[ 
  e^{F(\br )}\big| \cpl{l}{\br }=z
  \right] dz, \quad z\in \R , 
 \end{align*}
 it suffices to prove that the function 
 \begin{align}\label{;version}
 \R \ni z\mapsto \ex \left[ 
  e^{F(\br )}\big| \cpl{l}{\br }=z
  \right] 
 \end{align}
 admits an everywhere finite version that is log-concave in $z$. 
 To this end, define the path transform $w^{l},\,w\in \W $, by 
 \begin{align}\label{;ltrans}
  w^{l}(t):=w(t)-\cpl{l}{w}l(t), \quad 0\le t\le 1, 
 \end{align}
 where $l$ is regarded as an element in $\H $. Since 
 two Gaussians 
 $\left\{ \br ^{l}(t)\right\} _{0\le t\le 1}$ and 
 $\cpl{l}{\br }$ are uncorrelated, 
 they are independent, from which we have 
 \begin{align}\label{;cond}
  \ex \left[ 
  e^{F(\br )}\big| \cpl{l}{\br }=z
  \right] =\ex \left[ 
  e^{G(\br ,z)}
  \right] \quad \text{for a.e.\ }z\in \R , 
 \end{align}
 where we set $G(w,z):=F\left( w^{l}+zl\right) ,\,(w,z)\in \W \times \R $. 
 In view of \tref{;tmain2}, we show that this $G$ satisfies the 
 conditions \thetag{B1} and \thetag{B2}. It is clear that 
 $G$ satisfies \thetag{B2} thanks to the concavity of $F$ and 
 the linearity of the transformation \eqref{;ltrans}. 
 To see that \thetag{B1} is fulfilled, first note that by the 
 assumption \thetag{C3}, 
 \begin{align}\label{;gmbd}
  \ex \left[ 
  G_{-}(\br ,z)^{1+\delta }
  \right] <\infty 
 \end{align}
 for a.e.\ $z\in \R $, which readily follows by conditioning on 
 $\cpl{l}{\br }$ and using the independence noted above. We now show 
 that this a.e.\ finiteness can be extended to the everywhere 
 finiteness. For this purpose we fix $z_{0}\in \R $ arbitrarily. 
 Then we may find $z_{i}\in \R ,\,i=1,2$, such that 
 $
  z_{1}<z_{0}<z_{2}
 $ 
 and that each $z_{i}$ satisfies \eqref{;gmbd}. Since the 
 function $\R \ni r\mapsto (r\vee 0)^{1+\delta }$ is convex 
 and nondecreasing, and the function 
 $\R \ni z\mapsto -G(w,z)$ is convex for every fixed $w\in \W $, 
 their composition, namely 
 $G_{-}(w,z)^{1+\delta }=\left( (-G(w,z))\vee 0\right) ^{1+\delta }$, 
 is also convex in $z$, which entails that 
 \begin{align*}
  G_{-}(w,z_{0})^{1+\delta }\le 
  \theta G_{-}(w,z_{1})^{1+\delta }+(1-\theta )G_{-}(w,z_{2})^{1+\delta }
 \end{align*}
 for every $w\in \W $. Here 
 $\theta =(z_{2}-z_{0})/(z_{2}-z_{1})\in (0,1)$. 
 Taking the expectation in $w$ with respect to $\pr $ on both sides and 
 noting the finiteness \eqref{;gmbd} for $z_{i},\,i=1,2$, we obtain 
 \begin{align*}
  \ex \left[ 
  G_{-}(\br ,z_{0})^{1+\delta }
  \right] <\infty . 
 \end{align*}
 As $z_{0}\in \R $ is arbitrary, this shows that 
 $G$ satisfies the condition \thetag{B1} as well. 
 Therefore by \tref{;tmain2}, the function 
 $\R \ni z\mapsto \log \ex \left[ e^{G(\br ,z)}\right] $ 
 is concave. This function might take the value $\infty $, 
 but is finite a.e.\ by the assumption \thetag{C2} and the relation 
 \eqref{;cond}, which together with 
 concavity implies that it is in fact finite everywhere. 
 Consequently, the function \eqref{;version} 
 admits the everywhere finite log-concave version 
 $\ex \left[ e^{G(\br ,z)}\right] $. 
 The rest of the proof of the theorem proceeds in the same way 
 as in either \cite{bl}, \cite{caf}, or \cite{har}. 
\end{proof}

\appendix 
\section*{Appendix}
\renewcommand{\thesection}{A}
\setcounter{equation}{0}
\setcounter{prop}{0}
\setcounter{lem}{0}
\setcounter{rem}{0}

\ND\D{\mathcal{D}_{V}}
\ND\fm{F_{X}}
\ND\pdpinv{\Phi '\circ \Phi ^{-1}}
\ND\fdfinv{\fm '\circ \fm ^{-1}}
\ND\U{U_{V}}

In this appendix, we continue our discussion in 
\cite[Appendix]{har} as to an extension of the 
Brascamp-Lieb inequality \eqref{;bl0} to the case 
of nonconvex potentials  
and explore conditions on the potential function $V$ under 
which the inequality \eqref{;bl0} remains true. We restrict our 
exposition to one dimension; a remark on the 
multidimensional case will be given at the end of the 
appendix. 
The Brascamp-Lieb inequality has importance in the analysis 
of $\nabla \phi $ interface models with convex potentials and 
there has recently been growing a great interest in models with nonconvex 
potentials; see \cite{fs,gos,bk,cd} and references 
therein. 

Let $\nu $ be the normal distribution with mean $0$ and 
variance $\sigma ^{2},\,\sigma >0$, and let 
one-dimensional random variables $X$ and $Y$ be 
as given in \eqref{;laws}, in which we now suppose that 
the function $V:\R \to \R $ is in $C^{2}(\R )$ and not 
necessarily convex. We assume that $V$ is bounded from 
below by a linear function: 
\begin{align}\label{;bdlin}
 V(x)\ge ax+b \quad \text{for all }x\in \R , 
\end{align}
for some $a,b\in \R $, so that 
\[
 Z=E\left[ e^{-V(Y)}\right] <\infty . 
\]
We are interested in the case that 
$\left\{ x\in \R ;\,V''(x)<0\right\} \neq \emptyset $, 
which we will work in from now on. We denote 
\begin{align*}
 \D =\left\{ x\in \R ;\,V''(x)\le 0\right\} . 
\end{align*}
With these settings, the aim of 
this appendix is to give a proof of the 

\begin{prop}\label{;pncvx}
 Suppose that 
 \begin{align}\label{;inf1}
 \inf _{x\in \D }\left\{ 
 \frac{1}{2}\sigma ^2
 V'(x)^2+xV'(x)-V(x)
 \right\} \ge \log Z. 
 \end{align}
 Then it holds that for any convex function $\psi $ on $\R $, 
 \begin{align}\label{;bl1}
 E\left[ 
 \psi \left( X-E[X]\right) 
 \right] 
 \le E\left[ \psi (Y)\right] . 
 \end{align} 
 In particular, the same conclusion holds true if 
 \begin{align}\label{;inf2}
 \inf _{x\in \D }\left\{ 
 -\frac{x^2}{2\sigma ^2}-V(x)
 \right\} \ge \log Z. 
 \end{align}
\end{prop}

We give an example: 
\begin{exm}\label{;dw}
 Consider the potential $V$ of the form 
 \[
 V(x)=\frac{1}{2}\al ^2x^4-\frac{1}{2}\beta x^2, \quad x\in \R , 
 \]
 for $\al ,\beta >0$. Take $\sigma =1$ for simplicity. Then the 
 left-hand side of \eqref{;inf2} is calculated as 
 \[
 \frac{\beta (5\beta -6)}{72\al ^2}\wedge 0, 
 \]
 which tends to $0$ as $\al \to \infty $. On the other hand, 
 as 
 \begin{align*}
  Z=\frac{1}{\sqrt{2\pi \al }}
  \int _{\R }\exp \left( 
  \frac{\beta -1}{2\al }y^2-\frac{1}{2}y^4
  \right) dy 
 \end{align*}
 by change of variables, it is clear that the right-hand side 
 of \eqref{;inf2} diverges to $-\infty $ as $\al \to \infty $. 
 Therefore even if $\beta \gg 1$, the condition \eqref{;inf2} 
 is fulfilled by taking $\al $ sufficiently large, and hence 
 the inequality \eqref{;bl1} holds for such a pair of $\al $ 
 and $\beta $ by \pref{;pncvx}. 
\end{exm}

\begin{rem}
 As for the above example, the left-hand side of \eqref{;inf1} is 
 equal to 
 \begin{align*}
  \frac{\beta ^2(8\beta -9)}{216\al ^2}\wedge 0, 
 \end{align*}
 which gives a sharper condition on $\al $ and $\beta $ 
 for \eqref{;bl1} to hold. 
\end{rem}

We proceed to the proof of \pref{;pncvx}. In what follows we denote 
\begin{align*}
 \U (x)=\frac{1}{2}\sigma ^2
 V'(x)^2+xV'(x)-V(x), \quad x\in \R . 
\end{align*}
We also denote by $\fm $ the distribution function of the 
random variable $X$: 
\begin{align*}
 \fm (x)=\frac{1}{Z}\int _{-\infty }^{x}
 e^{-V(y)}\,\nu (dy), \quad x\in \R . 
\end{align*}
We define 
\begin{align}\label{;defg}
 g:=\fm ^{-1}\circ \Phi , 
\end{align}
where $\fm ^{-1}$ is the inverse function of $\fm $ and 
$\Phi $ is the standard normal cumulative distribution function: 
\begin{align*}
 \Phi (x)=\frac{1}{\sqrt{2\pi }}\int _{-\infty }^{x}
 \exp \left( -\frac{y^2}{2}\right) dy, \quad x\in \R . 
\end{align*}

\begin{lem}\label{;prepare}
 Suppose that for all $x\in \R $, 
 \begin{align}\label{;inf3}
  \U (x)\ge \log Z. 
 \end{align}
 Then the inequality \eqref{;bl1} holds for any 
 convex function $\psi $ on $\R $. 
\end{lem}

\begin{proof}
 In view of the proof of \cite[Theorem~1.1]{har}, it suffices to 
 show that 
 \begin{align}\label{;gbd}
  g'(x)\le \sigma \quad \text{for all }x\in \R . 
 \end{align}
 Indeed, if \eqref{;gbd} has been proven, then using Bass' solution 
 \cite{ba} to the Skorokhod embedding problem, one finds that 
 for a given one-dimensional standard Brownian motion 
 $B=\{ B(t)\} _{t\ge 0}$, there exists a stopping time 
 $T$ with respect to the natural filtration of $B$ such that 
 \begin{align*}
  X-E[X]\stackrel{(d)}{=}B(T) && \text{and} && 
  T\le \sigma ^2\quad \text{a.s. }
 \end{align*}
 Then the inequality \eqref{;bl1} is immediate either from 
 the optional sampling theorem applied to the 
 submartingale $\left\{ \psi (B(t))\right\} _{t\ge 0}$, 
 or from the It\^o-Tanaka formula. See \cite[Subsection~2.1]{har} 
 for details. 
 
 We turn to the proof of \eqref{;gbd}. The reasoning is the 
 same as in the proof of \cite[Lemma~2.1]{har}. Since 
 \begin{align*}
  g'(x)=\frac{\Phi '(x)}{\fdfinv \left( \Phi (x)\right) } 
 \end{align*}
 by the definition \eqref{;defg} of $g$, the inequality 
 \eqref{;gbd} is equivalent to 
 \begin{align}\label{;nonneg}
  G (\xi ):=\sigma \fdfinv (\xi )-\pdpinv (\xi )\ge 0 \quad 
  \text{for all }\xi \in (0,1). 
 \end{align}
 First note that 
 \begin{align}\label{;bdry}
  G(0+)=\lim _{\xi \to 0+}G(\xi )=0, 
  \quad G(1-)=\lim _{\xi \to 1-}G(\xi )=0 
 \end{align}
 because $\pdpinv $ satisfies 
 $\pdpinv (0+)=\pdpinv (1-)=0$ and so does 
 $\fdfinv $ by \eqref{;bdlin}. 
 We now suppose that $G$ has a local minimum at 
 some $\xi _{0}\in (0,1)$. Then, since 
  \begin{align*}
  G '(\xi )=-\left( \frac{x}{\sigma }+\sigma V '(x)\right) 
  \Big| _{x=\fm ^{-1}(\xi )}
  +\Phi ^{-1}(\xi ), \quad \xi \in (0,1), 
 \end{align*}
 we have 
 \begin{align*}
  \Phi ^{-1}(\xi _{0})=
  \left( \frac{x}{\sigma }+\sigma V '(x)\right) 
  \Big| _{x=\fm ^{-1}(\xi _{0})}. 
 \end{align*}
 Therefore by the definition of $G$, 
 \begin{align*}
  G (\xi _{0})&=\left\{ 
  \sigma \fm '(x)-\Phi '\left( \frac{x}{\sigma }+\sigma V '(x)\right) 
  \right\} \Big| _{x=\fm ^{-1}(\xi _{0})}\\
  &=\frac{1}{\sqrt{2\pi }}
  \exp \left( -\frac{x^2}{2\sigma ^2}-V(x)\right) 
  \left\{ 
  \frac{1}{Z}-\exp \left( -\U (x)\right) 
  \right\} \bigg| _{x=\fm ^{-1}(\xi _{0})}, 
 \end{align*}
 which is nonnegative by the assumption. 
 Combining this with \eqref{;bdry} shows 
 \eqref{;nonneg} and concludes the proof. 
\end{proof}

Using \lref{;prepare}, we prove \pref{;pncvx}

\begin{proof}[Proof of \pref{;pncvx}]
The latter assertion is immediate from the fact that 
\begin{align*}
 \U (x)&=\frac{1}{2}\left( 
 \sigma V'(x)+\frac{x}{\sigma}
 \right) ^2-\frac{x^2}{2\sigma ^2}-V(x)\\
 &\ge -\frac{x^2}{2\sigma ^2}-V(x)
\end{align*}
for all $x\in \R $. To show the former, 
take an arbitrary $x_{0}\in \R \backslash \D $, namely 
$x_{0}$ is such that $V''(x_{0})>0$. First we suppose that 
\begin{align*}
 V(x)>V'(x_{0})(x-x_{0})+V(x_{0}) 
\end{align*}
for all $x\in \R $ but $x_{0}$. 
Then as 
\begin{align*}
 Z&=\frac{1}{\sqrt{2\pi }\sigma }\int _{\R }\exp 
 \left( -\frac{x^2}{2\sigma ^2}-V(x)\right) dx\\
 &\le \frac{1}{\sqrt{2\pi }\sigma }
 \exp \left( x_{0}V'(x_{0})-V(x_{0})\right) 
 \int _{\R }\exp \left( -\frac{x^2}{2\sigma ^2}-V'(x_{0})x\right) dx\\
 &=\exp \left( \U (x_{0})\right) , 
\end{align*}
the inequality \eqref{;inf3} holds for $x=x_{0}$. 
Next we suppose that 
\begin{align*}
 V(x_{1})=V'(x_{0})(x_{1}-x_{0})+V(x_{0})
\end{align*}
for some $x_{1}\neq x_{0}$, say, $x_{1}>x_{0}$. Let 
$x_{2}\in [x_{0},x_{1}]$ be a maximal point of the 
function 
\begin{align*}
 f(x):=V(x)-V'(x_{0})(x-x_{0})-V(x_{0}), \quad 
 x\in [x_{0},x_{1}]. 
\end{align*}
Then it is clear that 
$f'(x_{2})=0$ and $f''(x_{2})\le 0$; 
indeed, if either of them were not the case, it would contradict 
the fact that $x_{2}$ is the maximal point. Therefore we have 
\begin{align}
 V'(x_{0})&=V'(x_{2}) \label{;veq}
\intertext{and }
 x_{2}&\in \D . \label{;indv}
\end{align}
Moreover, since 
\begin{align*}
 f(x_{2})=V(x_{2})-V'(x_{0})(x_{2}-x_{0})-V(x_{0})
 \ge f(x_{0})=0, 
\end{align*}
it also holds that by \eqref{;veq}, 
\begin{align*}
 x_{0}V'(x_{0})-V(x_{0})\ge x_{2}V'(x_{2})-V(x_{2}). 
\end{align*}
Combining this inequality with \eqref{;veq}, we have 
\begin{align*}
 \U (x_{0})&\ge \U (x_{2})\\
 &\ge \log Z, 
\end{align*}
where the second line is due to 
\eqref{;indv} and the assumption \eqref{;inf1}. 
Consequently, \eqref{;inf3} holds for all 
$x\in \R \backslash \D $, and hence for all $x\in \R $ 
by \eqref{;inf1}. Now the assertion of the proposition 
follows from \lref{;prepare}. 
\end{proof}

\begin{rem}\label{;rappend}
If one wants to apply the above discussion to the 
multidimensional case in order to extend the 
Brascamp-Lieb inequality \eqref{;bl0} to nonconvex potentials, 
it would be required to draw a condition on $V$ under which 
the function $\widetilde{V}$ defined by 
\begin{align*}
 \widetilde{V}(x)=
 -\log E\left[ 
 e^{-V(Y)}\big| \,\al \cdot Y=x
 \right] , \quad x\in \R , 
\end{align*}
fulfills either the assumption \eqref{;inf1} or \eqref{;inf2} of \pref{;pncvx} 
with $\sigma ^2=\al \cdot \Sigma \al $. 
Our original motivation to extend the variational 
representation \eqref{;vr0} to unbounded functionals stems from our 
desire to understand better the quantitative nature of 
$\widetilde{V}$ as well as that of the partition 
function $Z$. 
\end{rem}
\medskip 

\noindent 
{\bf Acknowledgements.} The author would like to thank Professor 
\"Ust\"unel for notifying him of two papers \cite{fu} and \cite{ust}. 


\end{document}